\documentclass[a4paper,10pt]{article}
\usepackage{amsmath,amssymb,amsthm,bm}
\bmdefine{\be}{e}
\bmdefine{\bg}{g}
\bmdefine{\bk}{k}
\bmdefine{\bm}{m}
\bmdefine{\bp}{p}
\bmdefine{\bs}{s}
\bmdefine{\bt}{t}
\bmdefine{\bw}{w}

\newtheorem{thm}{THEOREM}[section]
\newtheorem{lem}[thm]{LEMMA}
\newtheorem{cor}[thm]{COROLLARY}

\newtheorem{rem}{REMARK}

\newcommand{\Romannum}[1]{\uppercase\expandafter{\romannumeral#1}}
\makeatletter
  
  \@addtoreset{equation}{section}
\makeatother
\begin{document}
\title{Arithmetical properties of real numbers related to beta-expansions
}
\author{Kaneko Hajime 
}
\date{}
\maketitle
\begin{abstract}
The main purpose of this paper is to 
study the arithmetical properties of values \(\sum_{m=0}^{\infty} \beta^{-w(m)}\), 
where \(\beta\) is a fixed Pisot or Salem number and \(w(m)\) (\(m=0,1,\ldots\)) are distinct sequences of 
nonnegative integers with \(w(m+1)>w(m)\) for any sufficiently large \(m\). 
We first introduce criteria for the algebraic independence of such values. 
Our criteria are applicable to certain sequences \(w(m)\) (\(m=0,1,\ldots\)) with \(\lim_{m\to\infty}w(m+1)/w(m)=1.\) 
For example, we prove that two numbers 
\[\sum_{m=1}^{\infty}\beta^{-\lfloor \varphi(1,0;m)\rfloor}, 
\sum_{m=3}^{\infty}\beta^{-\lfloor \varphi(0,1;m)\rfloor}\] 
are algebraically independent, where 
\(\varphi(1,0;m)=m^{\log m}\) and \(\varphi(0,1;m)=m^{\log\log m}\). \par
Moreover, we also give criteria for linear independence of real numbers. 
Our criteria are applicable to the values 
\(\sum_{m=0}^{\infty}\beta^{-\lfloor m^\rho\rfloor}\), where 
\(\beta\) is a Pisot or Salem number and \(\rho\) is a real number greater than 1. 
\end{abstract}
\section{Introduction}
Throughout this paper, we denote the set of nonnegative integers (resp. positive integers) 
by \(\mathbb{N}\) (resp. \(\mathbb{Z}^{+}\)). We write the integral and fractional parts of a real number 
\(x\) by \(\lfloor x \rfloor\) and \(\{x\}\), respectively. Moreover, \(\lceil x \rceil\) 
is the minimal integer not less than \(x\). We use the Vinogradov symbols \(\gg\) and \(\ll\), 
as well as the Landau symbols \(O,o\) with their regular meanings. 
Finally, \(f\sim g\) means that the ratio \(f/g\) tends to \(1\) \par
In what follows, we investigate the arithmetical properties of the values of power series \(f(X)\) 
at algebraic points. For simplicity, we first consider the case where \(f(X)\) has the form 
\[f(X)=\sum_{m=0}^{\infty} X^{w(m)},\]
where \((w(m))_{m=0}^{\infty}\) is a sequence of nonnegative integers satisfying 
\(w(m)<w(m+1)\) for any sufficiently large \(m\). We call \(f(X)\) a gap series  
if 
\[\lim_{m\to\infty}\frac{w(m+1)}{w(m)}=\infty.\]
We say that \(f(X)\) is a lacunary series if 
\[\liminf_{m\to\infty}\frac{w(m+1)}{w(m)}>1.\]
Note that if \(f(X)\) is a lacunary series, then there exists a positive real number \(\delta\) such that 
\[w(m)>(1+\delta)^m\]
for any sufficiently large \(m\). \par
In the rest of this secction, suppose that \(\alpha\) is an algebraic number with \(0<|\alpha|<1\). 
In paper \cite{bug2}, Bugeaud posed a problem on the transcendence of the values of 
power series \(f(X)\) as follows: If \((w(m))_{m=0}^{\infty}\) increases sufficiently rapidly, then 
\(\sum_{m=0}^{\infty}\alpha^{w(m)}\) is transcendental. \par
Corvaja and Zannier \cite{cor} showed that if \(f(X)=\sum_{m=0}^{\infty} X^{w(m)}\) is a lacunary 
series, then \(\sum_{m=0}^{\infty}\alpha^{w(m)}\) is transcendental. For instance, let \(x, y\) be 
real numbers with \(x>0\) and \(y>1\). Then two numbers 
\[\sum_{m=0}^{\infty}\alpha^{\lfloor x (m!)\rfloor}
, 
\sum_{m=0}^{\infty}\alpha^{\lfloor y^m \rfloor}
\]
are transcendental. \par
Adamczewski \cite{ada1} improved the result above in the case of \(\alpha=\beta^{-1}\), where 
\(\beta\) is a Pisot or Salem number. Recall that Pisot numbers are algebraic integers greater than \(1\) 
whose conjugates except themselves have absolute values less than \(1\). 
Note that any rational integers greater than \(1\) are Pisot numbers. Salem numbers are 
algebraic integers greater than \(1\) such that the conjugates except themselves have moduli 
less than \(1\) and that there exists at least one conjugate with modulus \(1\). 
Adamczewski \cite{ada1} showed that if 
\[\liminf_{m\to\infty}\frac{w(m+1)}{w(m)}>1,\]
then \(\sum_{m=0}^{\infty}\beta^{-w(m)}\) is transcendental for any Pisot or Salem number \(\beta\). 
\par
We now introduce known results on the algebraic independence of certain lacunary series at 
fixed algebraic points. First we consider the case where \(f(X)\) is a gap series. Durand \cite{dur} 
showed that if \(\alpha\) is a real algebraic number with \(0<\alpha<1\), then the 
continuum set 
\begin{align}
\left\{\left. \sum_{m=0}^{\infty} \alpha^{\lfloor x (m!)\rfloor}\right| x\in\mathbb{R}, 
x>0\right\}
\label{eqn:int1}
\end{align}
is algebraically independent. Moreover, Shiokawa \cite{shi} gave a criterion for the algebraic independence 
of the values of certain gap series. Using his criterion, we deduce for general algebraic number 
\(\alpha\) with \(0<|\alpha|<1\) that the set (\ref{eqn:int1}) is algebraically independent. \par
Next, we consider the case where \(f(X)\) is not a gap series. 
Using Mahler's method for algebraic independence, Nishioka \cite{nis} proved that the set 
\begin{align*}
\left\{\left. \sum_{m=0}^{\infty} \alpha^{k^m}\right| k=2,3,\ldots\right\}
\end{align*}
is algebraically independent. Moreover, Tanaka \cite{tan} showed that if 
positive real numbers \(w_1,\ldots,w_m\) are linearly independent over \(\mathbb{Q}\), then 
the set 
\begin{align*}
\left\{\left. \sum_{m=0}^{\infty} \alpha^{\lfloor w_i k^m\rfloor}\right| i=1,\ldots,m, \ k=2,3,\ldots\right\}
\end{align*}
is algebraically independent. \par
On the other hand, it is generally difficult to study algebraic independence 
in the case where \(f(X)\) is not lacunary. 
In Section 2 we review known results on the criteria for transcendence of 
the value \(\sum_{m=0}^{\infty}\beta^{-w(m)}\), where \(\beta\) is a Pisot or Salem number and 
\((w(m))_{m=0}^{\infty}\) is a certain sequence of nonnegative integers with 
\begin{align*}
\lim_{m\to\infty}\frac{w(m+1)}{w(m)}=1.
\end{align*}
In Section 3 we give the main results on the algebraic independence of real numbers applicable to 
\[\sum_{m=1}^{\infty}\beta^{-\lfloor m^{\log m}\rfloor}, 
\sum_{m=3}^{\infty}\beta^{-\lfloor m^{\log \log m}\rfloor}.\]
In the same section we also investigate the linear independence of real numbers applicable to 
\(\sum_{m=0}^{\infty} \beta^{-\lfloor m^{\rho}\rfloor}\) for a real number \(\rho>1\). 
The main criteria for algebraic independence and linear independence, which are used to prove the main results, 
are denoted in Section 4.  
For the proof of the algebraic 
independence and linear independence, we need no functional equation because our criteria are flexible. 
We prove the main results in Section 5. Moreover, we show the criteria in Section 6. 
\section{Transcendental results related to the numbers of nonzero digits}
In this section we review criteria for the transcendence of the value \(\sum_{n=0}^{\infty} t_n \beta^{-n}\), 
where \((t_n)_{n=0}^{\infty}\) is a bounded sequence of nonnegative integers and \(\beta\) is a Pisot or 
Salem number. First we consider the case where \(\beta=b\) is an integer greater than \(1\). 
We denote the base-\(b\) expansion of a real number \(\eta\) by  
\[\eta=\sum_{n=0}^{\infty}s_n^{(b)}(\eta) b^{-n},\]
where \(s_0^{(b)}(\eta)=\lfloor \eta\rfloor\) 
and \(s_n^{(b)}(\eta)\in \{0,1,\ldots,b-1\}\) for any positive integer \(n\). 
We may assume that \(s_n^{(b)}(\eta)\leq b-2\) for infinitely many \(n\)'s. 
For any positive integer \(N\), put 
\[\lambda_b(\eta;N):=\mbox{Card}\{n\in\mathbb{N}\mid n< N, s_n^{(b)}(\eta)\ne 0\},\]
where Card denotes the cardinality. \par
Borel \cite{bor} conjectured for each integral base \(b\geq 2\) that any algebraic irrational number is 
normal in base-\(b\), which is still an open problem. For any real number \(\rho>1\), put 
\[\gamma(\rho;X):=\sum_{m=0}^{\infty}X^{\lfloor m^{\rho}\rfloor}.\]
If Borel's conjecture is true, then \(\gamma(\rho;b^{-1})\) is transcendental because 
\(\gamma(\rho;b^{-1})\) is a non-normal irrational number in base-\(b\). 
However, the transcendence of such values is 
not known except the case of \(\rho=2\). If \(\rho=2\), then 
Duverney, Nishioka, Nishioka, Shiokawa \cite{duv} and Bertrand \cite{ber} independently proved for 
any algebraic number \(\alpha\) with \(0<|\alpha|<1\) that 
\(\gamma(2;\alpha)\) is transcendental. \par
Bailey, Borwein, Crandall, and Pomerance \cite{bai} gave a criterion for the transcendence of 
real numbers, using lower bounds for the numbers of nonzero digits in the binary expansions of 
algebraic irrational numbers. Let \(\eta\) be an algebraic irrational number with degree \(D\). 
Bailey, Borwein, Crandall, and Pomerance \cite{bai} showed that there exist positive constants 
\(C_1(\eta)\) and \(C_2(\eta)\), depending only on \(\eta\), satisfying 
\begin{align*}
\lambda_2(\eta;N)\geq C_1(\eta) N^{1/D}
\end{align*}
for any integer \(N\) with \(N\geq C_2(\eta)\). Note that \(C_1(\eta)\) is effectively computable but 
\(C_2(\eta)\) is not. 
For any integral base \(b\geq 2\), 
Adamczewski, Faverjon \cite{ada2} and Bugeaud \cite{bug1} gave effective versions 
of lower bounds for \(\lambda_b(\eta;N)\) 
as follows: There exist effectively computable positive constants \(C_3(b,\eta)\) 
and \(C_4(b,\eta)\), depending only on \(b\) and \(\eta\), satisfying 
\begin{align}
\lambda_b(\eta;N)\geq C_3(b,\eta) N^{1/D}
\label{eqn:tra2}
\end{align}
for any integer \(N\) with \(N\geq C_4(b,\eta)\). Using (\ref{eqn:tra2}), we obtain for any real number 
\(\rho>1\) that \(\gamma(\rho;b^{-1})\) is not an algebraic number of degree less than \(\rho\). In fact, 
\(\gamma(\rho;b^{-1})\) is an irrational number satisfying 
\[\lambda_b\bigl(\gamma(\rho;b^{-1});N\bigr)\sim N^{1/{\rho}}\]
as \(N\) tends to infinity. Thus, (\ref{eqn:tra2}) does not hold if \(D<\rho\). \par
By (\ref{eqn:tra2}), we also deduce a criterion for the transcendence of real numbers as follows: 
Let \(\eta\) be a positive irrational number. Suppose for any real positive real number \(\varepsilon\) 
that 
\begin{align}
\liminf_{N\to\infty}\frac{\lambda_b(\eta;N)}{N^{\varepsilon}}=0.
\label{eqn:tra3}
\end{align}
Then \(\eta\) is a transcendental number. 
Note that the criterion above was essentially obtained by 
Bailey, Borwein, Crandall, and Pomerance \cite{bai}. 
Note that if \(\sum_{m=0}^{\infty} X^{w(m)}\) is lacunary, then \(\eta=\sum_{m=0}^{\infty} b^{-w(m)}\) 
satisfies (\ref{eqn:tra3}) by 
\[\lambda_b(\eta;N)=O(\log N).\]
We give another example  of transcendental numbers. For any real numbers \(y>0\) and 
\(R\geq 1\), we put 
\[\varphi(y;R):=\exp\left((\log R)^{1+y}\right)=R^{(\log R)^y}.\]
Moreover, we set 
\[\xi(y;X):=1+\sum_{m=1}^{\infty} X^{\lfloor \varphi(y;m)\rfloor}.\]
Note that \(\xi(y;X)\) is not lacunary by 
\[\lim_{m\to\infty}\frac{\varphi(y;m+1)}{\varphi(y;m)}=1.\]
We get that \(\eta:=\xi(y;b^{-1})\) is transcendental for any 
integer \(b\geq 2\) because \(\eta\) satisfies (\ref{eqn:tra3}). \par 
In what follows, we consider the case where 
\(\beta\) is a general Pisot or Salem number. 
We introduce results in \cite{kan3} related to the \(\beta\)-expansion of algebraic numbers. 
For any formal power series \(f(X)=\sum_{n=0}^{\infty} t_n X^n\), we put 
\[S(f):=\{n\in\mathbb{N}\mid t_n\ne 0\}. \]
Moreover, for any nonempty set \(\mathcal{A}\) of nonnegative integers, we set 
\[\lambda(\mathcal{A};N):=\mbox{Card}(\mathcal{A}\cap [0,N)).\]
We denote the degree of a field extension \(L/K\) by \([L:K]\). 
\begin{thm}[\cite{kan3}]
Let \(A\) be a positive integer and 
let \(f(X)=\sum_{n=0}^{\infty}t_n X^n\) be a power series with integral coefficients. 
Assume that \(0\leq t_n\leq A\) for any nonnegative integer \(n\) and that there exist 
infinitely many \(n\)'s satisfying \(t_n\ne 0\). Let \(\beta\) be a Pisot or Salem number. 
Suppose that \(\eta=f(\beta^{-1})\) is an algebraic number with 
\([\mathbb{Q}(\beta,\eta):\mathbb{Q}(\beta)]=D\). Then there exist effectively computable positive 
constants \(C_5(A,\beta,\eta)\) and \(C_6(A,\beta,\eta)\), depending only on \(A,\beta\) and \(\eta\) 
satisfying 
\[\lambda\bigl(S(f);N\bigr)\geq C_5(A,\beta,\eta)\left(\frac{N}{\log N}\right)^{1/D}\]
for any integer \(N\) with \(N\geq C_6(A,\beta,\eta)\). 
\label{thm:2-1}
\end{thm}
In the rest of this section, let \(\beta\) be a Pisot or Salem number. 
Using Theorem \ref{thm:2-1}, we obtain for any real number \(\rho>1\) that 
\[\left[\mathbb{Q}\bigl(\gamma(\rho;\beta^{-1}),\beta\bigr):\mathbb{Q}(\beta)\right]\geq \lceil \rho\rceil\]
by 
\begin{align}
\lambda(S(\gamma(\rho;X));N)\sim N^{1/\rho}
\label{eqn:tra4}
\end{align} as \(N\) tends to infinity. \par
Note that Theorem \ref{thm:2-1} is applicable to the study of the nonzero digits in the \(\beta\)-expansions 
of algebraic numbers. We recall the definition of \(\beta\)-expansion defined by R\'{e}nyi \cite{ren} in 1957. 
Let \(T_{\beta}:[0,1)\to[0,1)\) be the \(\beta\)-transformation defined by 
\(T_{\beta}(x)=\{\beta x\}\) for \(x\in[0,1)\). Then the \(\beta\)-expansion of a real number 
\(\eta\in[0,1)\) is denoted as 
\[\eta=\sum_{n=1}^{\infty} s_n^{(\beta)}(\eta)\beta^{-n},\]
where \(s_n^{(\beta)}(\eta)=\lfloor \beta T_{\beta}^{n-1}(\eta)\rfloor \) for any \(n\geq 1\). 
Note that \(0\leq s_n^{(\beta)}(\eta)\leq \lfloor\beta\rfloor\) for any \(n\geq 1\). Put 
\[\lambda_{\beta}(\eta;N):=\mbox{Card}\{n\in\mathbb{Z}^{+}, n\leq N, s_n^{(\beta)}(\eta)\ne 0\}\]
for any positive integer \(N\). Applying Theorem \ref{thm:2-1} with \(B=\lfloor\beta\rfloor\), 
we deduce that if \(\eta\in[0,1)\) is an algebraic number 
with \([\mathbb{Q}(\beta,\eta):\mathbb{Q}(\beta)]=D\), then 
\[\lambda_{\beta}(\eta;N)\gg\left(\frac{N}{\log N}\right)^{1/D}\]
for any sufficiently large integer \(N\). \par
Using Theorem \ref{thm:2-1}, we also deduce a criterion for the transcendence of real numbers 
as follows: 
Let \(f(X)\) be a power series whose coefficients are bounded nonnegative integers. 
Suppose that \(f(X)\) is not a polynomial and that 
\[\liminf_{m\to\infty}\frac{\lambda_{\beta}(S(f);N)}{N^{\varepsilon}}=0\]
for any positive real number \(\varepsilon\). Then \(f(\beta^{-1})\) is transcendental. Note that 
the criterion above was already obtained in \cite{kan2} 
and that the criterion is applicable even if the representation 
\(\sum_{n=0}^{\infty} t_n \beta^{-n}\) does not coincide with the \(\beta\)-expansion 
of \(f(\beta^{-1})\). In the same way as the case where 
\(\beta=b\geq 2\) is an integer, we obtain for any positive real number \(y\) that 
\(\xi(y;\beta^{-1})\) is transcendental. \par
In the end of this section we introduce a corollary of Theorem \ref{thm:2-1}, which we need to 
prove our criteria for linear independence. 
\begin{cor}
Let \(A\) be a positive integer and \(f(X)\) a nonpolynomial power series whose coefficients are 
bounded nonnegative integers. 
Assume that there exists a positive real number \(\delta\) satisfying 
\[\lambda\bigl(S(f);R\bigr)<R^{-\delta+1/A}\]
for infinitely many integer \(R\geq 0\). 
Then, for any Pisot or Salem number \(\beta\), we have 
\[\left[\mathbb{Q}\bigl(f(\beta^{-1}),\beta\bigr):\mathbb{Q}(\beta)\right]\geq A+1.\]
\label{cor:tra1}
\end{cor}
\section{Main results}
\subsection{Results on algebraic independence}
We use the same notation as Section 2. 
\begin{thm}
Let \(\beta\) be a Pisot or Salem number. Then the continuum set 
\begin{align}
\{\xi(y;\beta^{-1})\mid y\in\mathbb{R}, \ y\geq 1\}
\label{eqn:mai1}
\end{align}
is algebraically independent. 
\label{thm:mai1}
\end{thm}
Note that if \(\beta=b\) is an integer greater than 1, then the algebraic independence of (\ref{eqn:mai1}) 
was proved in \cite{kan1}. However, the algebraic independence of the set 
\[\{\xi(y;b^{-1})\mid y\in\mathbb{R}, \ y>0\}\]
is unknown. \par
On the other hand, considering the algebraic independence of two values, we obtain 
more detailed results. 
Set 
\[{\Theta}:=\{(y,z)\in\mathbb{R}^2\mid y>0\mbox{, or }y=0\mbox{ and }z>0\}.\]
Moreover, for any real number \(R\geq 3\) and \((y,z)\in {\Theta}\), we put 
\begin{align*}
\varphi(y,z;R)& :=\exp \left((\log R)^{1+y} (\log\log R)^z\right)\\
&=R^{(\log R)^y (\log\log R)^z}
\end{align*}
and 
\begin{align*}
\xi(y,z;X):=1+\sum_{m=3}^{\infty} X^{\lfloor \varphi(y,z;m)\rfloor}.
\end{align*}
\begin{thm}
Let \((y_1,z_1)\) and \((y_2,z_2)\) be distinct elements in \({\Theta}\). Then the two values 
\(\xi(y_1,z_1;\beta^{-1})\) and \(\xi(y_2,z_2;\beta^{-1})\) are algebraically independent 
for any Pisot or Salem number \(\beta\). 
\label{thm:mai2}
\end{thm}
Considering the case of \(z_1=z_2=0\) in Theorem \ref{thm:mai2}, we get the following: 
\begin{cor}
Let \(y_1\) and \(y_2\) be distinct positive real numbers. Then the two values 
\(\xi(y_1;\beta^{-1})\) and \(\xi(y_2;\beta^{-1})\) are algebraically independent 
for any Pisot or Salem number \(\beta\). 
\label{cor:mai1}
\end{cor}
In the case where \(\beta=b\) is an integer greater than 1, the algebraic independence of the two values 
\(\xi(y_1;b^{-1})\) and \(\xi(y_2;b^{-1})\) was obtained in \cite{kan1}. \par 
Applying Theorem \ref{thm:mai2} with \((y_1,z_1)=(1,0)\) and \((y_2,z_2)=(0,1)\), we deduce the following: 
\begin{cor}
For any Pisot or Salem number \(\beta\) the two values 
\[\sum_{m=1}^{\infty} \beta^{-\lfloor m^{\log m}\rfloor}, 
\sum_{m=3}^{\infty} \beta^{-\lfloor m^{\log\log m}\rfloor}\]
are algebraically independent. 
\label{cor:mai2}
\end{cor}
In the last of this subsection, we introduce the algebraic independence of the values of \(\xi(y,z;X)\) and 
lacunary series. 
\begin{thm}
Let \((y,z)\in \Xi\) and let \(x\) be a real number greater than 1. Then, 
\(\xi(y,z,\beta^{-1})\) and \(\sum_{m=0}^{\infty} \beta^{-\lfloor x^m\rfloor}\) are 
algebraically independent for any Pisot or Salem number \(\beta\). 
\label{thm:abc}
\end{thm}
\subsection{Results on linear independence}
Let \(\mathcal{F}\) be the set of nonpolynomial power series 
\(g(X)\) satisfying the following three assumptions: 
\begin{enumerate}
\item The coefficients of \(g(X)\) are bounded 
nonnegative integers. 
\item For an arbitrary positive real number \(\varepsilon\), we have 
\[\lambda(S(g);R)=o(R^{\varepsilon})\]
as \(R\) tends to infinity. 
\item There exists a positive constant \(C\) such that 
\[[R, CR]\cap S(g)\ne \emptyset\]
for any sufficiently large \(R\). 
\end{enumerate}
In order to state our results, we give a lemma on the zeros of certain polynomials. 
For any positive integer \(k\), put 
\[G_k(X):=(1-X)^k+(k-1)X-1.\]
\begin{lem}
Suppose that \(k\geq 3\). Then the following holds: \\
\(\mathrm{1)}\) There exists a unique zero \(\sigma_k\) of \(G_k(X)\) 
on the interval \((0,1)\). \\
\(\mathrm{2)}\) Let \(x\) be a real number with \(0<x<1\). Then 
\(G_k(x)<0\) (resp. \(G_k(x)>0\)) if and only if \(x<\sigma_k\) (resp. \(x>\sigma_k\)). \\
\(\mathrm{3)}\) \((\sigma_k)_{k=3}^{\infty}\) is strictly decreasing. 
\label{lem:mai1}
\end{lem}
\begin{proof}
Observe that \(G_k'(X)=-k(1-X)^{k-1}+k-1\) is monotone increasing on the interval \((0,1)\) 
and that \(G_k'(X)\) has a unique zero \(\widetilde{\sigma_k}\) on \((0,1)\). 
Thus, \(G_k(X)\) is monotonically decreasing on \((0,\widetilde{\sigma_k}]\) and 
monotonically increasing on \((\widetilde{\sigma_k},1)\). Hence, the first and second statements of the 
lemma follow from \(G_k(0)=0\) and \(G_k(1)=k-2>0\). \par
Next, we assume that \(k\geq 4\). Using 
\[G_{k-1}(\sigma_{k-1})=(1-\sigma_{k-1})^{k-1}+(k-2)\sigma_{k-1}-1=0,\]
we get 
\[G_{k}(\sigma_{k-1})=(1-\sigma_{k-1})^k+(k-1)\sigma_{k-1}-1=(k-2)\sigma_{k-1}^2>0. \]
Hence, we obtain \(\sigma_k<\sigma_{k-1}\) by the second statement of the lemma. 
\end{proof}
\begin{thm}
Let \(A\) be a positive integer and \(\rho\) a real number. 
Suppose that 
\begin{align}
\left\{
\begin{array}{cc}
\rho> A & \mbox{ if }A\leq 3,\\
\rho> \sigma_A^{-1} & \mbox{ if }A\geq 4.
\end{array}
\right.
\label{eqn:mai2}
\end{align}
Then, for any \(g(X)\in \mathcal{F}\) and any 
Pisot or Salem number \(\beta\), the set 
\[\{\gamma(\rho;\beta^{-1})^{k_1} g(\beta^{-1})^{k_2}\mid k_1,k_2\in\mathbb{N}, k_1\leq A\}\]
is linearly independent over \(\mathbb{Q}(\beta)\). 
\label{thm:mai3}
\end{thm}
We give numerical examples of \(\sigma_n^{-1}\) (\(n\geq 4\)) as follows: 
\[\sigma_4^{-1}=5.278\ldots , \ 
\sigma_5^{-1}=8.942\ldots , \ \sigma_6^{-1}=13.60\ldots.\] 
\begin{cor}
Let \(A,\rho\) be as in Theorem \ref{thm:mai3}. \\
\(\mathrm{1)}\) For any real number \(y>1\) and any Pisot or Salem number \(\beta\), the set 
\[\left\{
\left.\gamma(\rho;\beta^{-1})^{k_1} \left(\sum_{m=0}^{\infty}\beta^{-\lfloor y^m\rfloor}\right)^{k_2} \ \right| \ 
k_1,k_2\in\mathbb{N}, k_1\leq A\right\}\]
is linearly independent over \(\mathbb{Q}(\beta)\). \\
\(\mathrm{2)}\) For any \((y,z)\in {\Theta}\) and any Pisot or Salem number \(\beta\), the set 
\[\{\gamma(\rho;\beta^{-1})^{k_1} \xi(y,z;\beta^{-1})^{k_2}\mid k_1,k_2\in\mathbb{N}, k_1\leq A\}\]
is linearly independent over \(\mathbb{Q}(\beta)\). 
\label{cor:mai3}
\end{cor}
Using the asymptotic behavior of the sequence \((\sigma_m)_{m=3}^{\infty}\), we deduce the following: 
\begin{cor}
Let \(\varepsilon\) be an arbitrary positive real number. Then there exists an effectively 
computable positive constant \(A_0(\varepsilon)\), depending only on \(\varepsilon\) 
satisfying the following: 
Let \(A\) be an integer with \(A\geq A_0(\varepsilon)\) and \(\rho\) a real number with 
\(\rho>(\varepsilon+1/2)A^2\). Then, for any \(g(X)\in \mathcal{F}\) and any 
Pisot or Salem number \(\beta\), the set 
\[\{\gamma(\rho;\beta^{-1})^{k_1} g(\beta^{-1})^{k_2}\mid k_1,k_2\in\mathbb{N}, k_1\leq A\}\]
is linearly independent over \(\mathbb{Q}(\beta)\). 
\label{cor:mai4}
\end{cor}
\section{Criteria for algebraic independence and linear independence}
Let \(k\) be a nonnegative integer and \(f(X)\in\mathbb{Z}[[X]]\backslash\mathbb{Z}[X]\). 
We denote the Minkowski sum of \(S(f)\) by 
\begin{align*}
k S(f):=
\left\{
\begin{array}{cc}
\{0\} & (k=0),\\
\{s_1+\cdots+s_k\mid s_1,\ldots,s_k\in S(f)\} & (k\geq 1).
\end{array}
\right.
\end{align*}
Moreover, for any \((k_1,\ldots,k_r)\in\mathbb{N}^r\) and 
\(f_1(X),\ldots,f_r(X)\in\mathbb{Z}[[X]]\backslash\mathbb{Z}[X]\), we set 
\begin{align*}
\sum_{h=1}^r k_h S(f_h):=\{s_1+\cdots+s_r\mid s_h\in k_h S(f_h) \mbox{ for }h=1,\ldots,r\}.
\end{align*}
\begin{rem}
\begin{rm}
Suppose that \(0\in S(f_i)\) for \(i=1,\ldots,r\). Then, for any 
\((k_1,\ldots,k_r)\in\mathbb{N}^r\) and \((k_1',\ldots,k_r')\in\mathbb{N}^r\) with 
\(k_i\geq k_i'\) for any \(i=1,\ldots,r\), we have 
\[\sum_{h=1}^r k_h S(f_h)\supset \sum_{h=1}^r k_h' S(f_h).\]
\label{rem:cri1}
\end{rm}
\end{rem}
Let \(\mathcal{A}\) be a nonempty set of nonnegative integers and 
\(R\) a real number with \(R>\min \mathcal{A}\). Then we put 
\[\theta(R;\mathcal{A}):=\max\{n\in\mathcal{A}\mid n<R\}.\]
\begin{thm}
Let \(A,r\) be integers with \(A\geq 1\) and \(r\geq 2\). 
Let \(f_i(X)=\sum_{n=0}^{\infty} t_i(n) X^n (i=1,\ldots,r)\) be nonpolynomial power series with 
integral coefficients. 
We assume that \(f_1(X),\ldots,f_r(X)\) satisfy the following four assumptions: 
\begin{enumerate}
\item There exists a positive constant \(C_7\) satisfying 
\[0\leq t_i(n)\leq C_7\]
for any \(i=1,\ldots, r\) and nonnegative integer \(n\). 
\item Let \(k_1,\ldots,k_r\) be nonnegative integers. 
Suppose that 
\begin{align}
\left\{
\begin{array}{cc}
k_1\leq A-1 & \mbox{ if }r=2,\\
k_1\leq A & \mbox{ if }r\geq 3.
\end{array}
\right.
\label{eqn:cri1}
\end{align}
Then 
\begin{align}
&R-\theta\left(R;\sum_{h=1}^{r-2} k_h S(f_h)+(1+k_{r-1})S(f_{r-1})\right)\nonumber\\
&\hspace{35mm}=o\left(\frac{R}{\prod_{h=1}^r \lambda(S(f_h);R)^{k_h}}\right)
\label{ppp}
\end{align}
as \(R\) tends to infinity. 
\item There exists a positive real number \(\delta\) satisfying 
\begin{align*}
\lambda(S(f_1);R)=o\left(R^{-\delta+1/A}\right)
\end{align*}
as \(R\) tends to infinity. Moreover, for any \(i=2,\ldots,r\) and any real number \(\varepsilon\), we have 
\[\lambda(S(f_i);R)=o\Bigl(\lambda\bigl(S(f_{i-1});R\bigr)^{\varepsilon}\Bigr)\]
as \(R\) tends to infinity. 
\item There exist positive constants \(C_8,C_9\) such that 
\[[R, C_8 R]\cap S(f_r)\ne \emptyset\]
for any real number \(R\) with \(R\geq C_9\). 
\end{enumerate}
Then, for any Pisot or Salem number \(\beta\), the set 
\[\{f_1(\beta^{-1})^{k_1}f_2(\beta^{-1})^{k_2}\cdots f_r(\beta^{-1})^{k_r}\mid 
k_1,k_2,\ldots,k_r\in\mathbb{N}, k_1\leq A\}\]
is linearly independent over \(\mathbb{Q}(\beta)\). 
\label{thm:cri1}
\end{thm}
Let \(a(R)\) be a real valued function defined on an interval \([R_0,\infty)\) with \(R_0\in\mathbb{R}\). 
We say that 
\(a(R)\) ultimately increasing if 
\(a(R)\) is strictly increasing for any sufficiently large real number \(R\). 
Similarly, we say that \((a(m))_{m=m_0}^{\infty}\) is ultimately increasing if this sequence is strictly 
increasing for any sufficiently large integer \(m\). 
\begin{thm}
Let \(a(R),u(R)\) be ultimately increasing functions defined on \([m_0,\infty)\) with \(m_0\in\mathbb{N}\). 
Assume that \((\lfloor a(m)\rfloor)_{m=m_0}^{\infty}\) and 
\((\lfloor u(m)\rfloor)_{m=m_0}^{\infty}\) are also 
ultimately increasing. 
Let \(b(R),v(R)\) be the inverse functions of \(a(R),u(R)\), respectively, for any sufficiently large \(R\). 
Assume that \(a(R)\) satisfies the following two assumptions: 
\begin{enumerate}
\item \((\log a(R))/(\log R)\) is ultimately increasing and 
\begin{align}\lim_{R\to\infty}\frac{\log a(R)}{\log R}=\infty.
\label{qqq}
\end{align}
\item We have \(a(R)\) is differentiable. Moreover, for an arbitrary positive real number \(\varepsilon\), 
there exists a positive constant \(C_{10}(\varepsilon)\), depending only on \(\varepsilon\), such that 
\[(\log a(R))'<R^{-1+\varepsilon}\]
for any real number \(R\) with \(R\geq C_{10}(\varepsilon)\). 
\end{enumerate}
Moreover, suppose that \(u(R)\) fulfills the following two assumptions: 
\begin{enumerate}
\item There exists a positive constant \(C_{11}\) such that 
\[\frac{u(R+1)}{u(R)}<C_{11}\]
for any sufficiently large real number \(R\). 
\item 
\begin{align}
\lim_{R\to\infty}\frac{\log b(R)}{\log v(R)}=\infty.
\label{eqn:cri3}
\end{align}
Then, for any Pisot or Salem number \(\beta\), the two numbers 
\[\sum_{m=m_0}^{\infty}\beta^{-\lfloor a(m)\rfloor}, \sum_{m=m_0}^{\infty}\beta^{-\lfloor u(m)\rfloor}\]
are algebraically independent. 
\end{enumerate}
\label{thm:cri2}
\end{thm}
\section{Proof of main results}
In this section we prove results in Section 3, using Theorems \ref{thm:cri1} and \ref{thm:cri2}. 
\subsection{Proof of results on algebraic independence}
\begin{proof}[Proof of Theorem \ref{thm:mai1}]
Let \(y_1,y_2,\ldots,y_r\) be real numbers with \(1\leq y_1<y_2<\cdots<y_r\). 
We show that 
\(f_i(X):=\xi(y_i;X) \ (i=1,\ldots,r)\)
fulfill the assumptions in Theorem \ref{thm:cri1} for any positive integer \(A\). 
The first assumption is clear. 
Recall that we proved Theorem 1.3 in \cite{kan1}, showing for any integer \(b \geq 2\) that 
\(f_1(b^{-1}),\ldots,f_r(b^{-1})\) satisfy the assumptions of Theorem 2.1 in \cite{kan1}. 
In the same way, we can check that \(f_1(X),\ldots,f(X)\) fulfill the third and fourth assumptions in 
Theorem \ref{thm:cri1}. \par
In what follows, we verify the second assumption. 
Let \(y\) be a fixed positive real number. Then we denote the inverse function of 
\(\varphi(y;R)\) by 
\[\psi(y;R)=\exp\left((\log R)^{1/(1+y)}\right).\]
For \(i=1,\ldots,r\), we have 
\[\lambda\bigl(S(f_i);R\big)\sim \psi(y_i;R)\]
as \(R\) tends to infinity. 
\begin{lem}
Let \(\bk=(k_1,\ldots,k_r)\in\mathbb{N}^r\backslash \{(0,\ldots,0)\}\). Then 
\[R-\theta\left(R;\sum_{i=1}^r k_i S(f_i)\right)\ll 
\frac{R(\log R)^{k_1+\cdots+k_r}}{\prod_{i=1}^r \psi(y_i;R)^{k_i}}\]
for any real \(R\) with \(R\geq 2\). 
\label{lem5:1}
\end{lem}
\begin{proof}
We can show Lemma \ref{lem5:1} in the same way as the proof of Lemma 3.1 in \cite{kan1}. 
\end{proof}
Let \(A\) be any positive integer and \(k_1,\ldots,k_r\) any nonnegative integers. 
Without loss of generality, we may assume that \(k_r\geq 1\). 
Applying Lemma \ref{lem5:1} with 
\(\bk=(k_1,\ldots,k_{r-2},1+k_{r-1},0)\in\mathbb{N}^r\backslash \{(0,\ldots,0)\}\), 
we get for any \(R\geq 2\) that 
\begin{align}
&R-\theta\left(R;\sum_{h=1}^{r-2} k_h S(f_h)+(1+k_{r-1})S(f_{r-1})\right)\nonumber\\
&\hspace{30mm}=o\left(\frac{R(\log R)^{1+k_1+\cdots+k_{r-1}}}
{\psi(y_{r-1};R)\prod_{h=1}^{r-1} \psi(y_{i};R)^{k_i}}\right).
\label{eqn5:2}
\end{align}
Observe that 
\begin{align*}
\log \left((\log R)^{1+k_1+\cdots+k_{r-1}}\right)
& \ll \log\log R\\
& =o\left((\log R)^{1/{(1+y_{r-1})}}\right)\\
&=o\left(\frac12\log \psi(y_{r-1};R)\right)
\end{align*}
as \(R\) tends to infinity. Thus, we see 
\begin{align}
(\log R)^{1+k_1+\cdots+k_{r-1}}=o\left(\psi(y_{r-1};R)^{1/2}\right).
\label{eqn5:3}
\end{align}
Combining (\ref{eqn5:2}) and (\ref{eqn5:3}), we obtain 
\begin{align*}
&R-\theta\left(R;\sum_{h=1}^{r-2} k_h S(f_h)+(1+k_{r-1})S(f_{r-1})\right)\\
&=
o\left(\frac{R}
{\psi(y_{r-1};R)^{1/2}\prod_{h=1}^{r-1} \psi(y_{i};R)^{k_i}}\right)
=
o\left(\frac{R}
{\prod_{h=1}^{r} \psi(y_{i};R)^{k_i}}\right),
\end{align*}
where we use the third assumption 
in Theorem \ref{thm:cri1} with \(i=r\) and \(\varepsilon=1/(2k_r)\) for the last equality. Therefore, we checked the second assumption. 
\end{proof}
\begin{proof}[Proof of Theorem \ref{thm:mai2}]
Without loss of generality, we may assume that \(y_1<y_2\), or \(y_1=y_2\) and \(z_1<z_2\). 
Put 
\[a(R):=\varphi(y_1,z_1;R), u(R):=\varphi(y_2,z_2;R).\]
In what follows, we check that 
\(a(R),u(R)\) satisfy the assumptions in Theorem \ref{thm:cri2}. Note that 
 \(a(R), u(R)\ (R\geq 3) \) and 
\((\lfloor a(m)\rfloor)_{m=3}^{\infty}\), \((\lfloor u(m)\rfloor)_{m=3}^{\infty}\) are 
ultimately increasing. The assumptions on \(a(R)\) in Theorem \ref{thm:cri2} are easily checked. 
In fact, the first assumption holds by 
\begin{align*}
\frac{\log a(R)}{\log R}=(\log R)^{y_1} (\log \log R)^{z_1}.
\end{align*}
Moreover, the second assumption follows from 
\begin{align}
&(\log a(R))'\nonumber\\
& =
\left\{
\begin{array}{cc}
(1+y_1)(\log R)^{y_1}/R & \mbox{ if } z_1=0, \\
(\log R)^{y_1}(\log \log R)^{-1+z_1}\bigl(z_1+(1+y_1)\log \log R\bigr)/R & \mbox{ if }z_1\ne 0.
\end{array}
\right. 
\label{eqn5:5}
\end{align}
Calculating \((\log u(R))'\) in the same way as (\ref{eqn5:5}), we see 
\[\lim_{R\to\infty}(\log u(R))'=0.\]
Using the mean value theorem, we get 
\begin{align}
\lim_{R\to\infty}\frac{u(R+1)}{u(R)}=1,
\label{eqn5:6}
\end{align}
which implies the first assumption on \(u(R)\) in Theorem \ref{thm:cri2}. \par
We now check the second assumption on \(u(R)\). Using 
\[\log a(R)=(\log R)^{1+y_1} (\log \log R)^{z_1},\]
we get 
\begin{align}
\log R=(\log b(R))^{1+y_1} (\log \log b(R))^{z_1}.
\label{eqn5:7}
\end{align}
Similarly, 
\begin{align}
\log R=(\log v(R))^{1+y_2} (\log \log v(R))^{z_2}.
\label{eqn5:8}
\end{align}
First we assume that \(y_1<y_2\). Put \(d:=y_2-y_1>0\). By (\ref{eqn5:7}) and (\ref{eqn5:8}), we get 
\[(\log v(R))^{1+y_1+(2d)/3}<\log R< (\log b(R))^{1+y_1+d/3}\]
for any sufficiently large \(R\). 
Consequently, we obtain 
\[(\log v(R))^{d/3}<\left(\frac{\log b(R)}{\log v(R)}\right)^{1+y_1+d/3},\]
which implies (\ref{eqn:cri3}). \par
Next we assume that \(y_1=y_2=:y\) and \(z_1<z_2\). Using (\ref{eqn5:7}) and (\ref{eqn5:8}) again, 
we see 
\begin{align}
\frac{(\log\log v(R))^{z_2}}{(\log\log b(R))^{z_1}}
=\left(\frac{\log b(R)}{\log v(R)}\right)^{1+y}.
\label{eqn5:9}
\end{align}
Taking the logarithm of the both-hand sides of (\ref{eqn5:9}), we get 
\begin{align}
&z_2\log\log\log v(R)-z_1\log\log\log b(R)\nonumber\\
&\hspace{20mm}=
(1+y)\log\log b(R) - (1+y)\log\log v(R)
\label{eqn5:10}
\end{align}
Note that \(b(R)\geq v(R)\) for any sufficiently large \(R\). Thus, dividing (\ref{eqn5:10}) by 
\(\log\log b(R)\), we see 
\begin{align}
\lim_{R\to\infty}\frac{\log \log v(R)}{\log \log b(R)}=1.
\label{eqn5:11}
\end{align}
Combining (\ref{eqn5:9}), (\ref{eqn5:11}), and \(z_2>z_1\), we deduce (\ref{eqn:cri3}). 
\end{proof}
\begin{proof}[Proof of Theorem \ref{thm:abc}]
Applying Theorem \ref{thm:cri2} with 
\[a(R):=\varphi(y,z;R), u(R):=x^R,\]
we deduce Theorem \ref{thm:abc}. 
In fact, we can check the assumptions on \(a(R)\) in 
Theorem \ref{thm:cri2} in the same way as the proof of 
Theorem \ref{thm:mai2}. 
Moreover, (\ref{eqn:cri3}) is seen by (\ref{eqn5:7}) and 
\(v(R)=(\log R)/(\log x).\)
\end{proof}
\subsection{Proof of results on linear independence}
\begin{proof}[Proof of Theorem \ref{thm:mai3}]
We show that the assumptions on Theorem \ref{thm:cri1} are satisfied, where \(A\) is defined as in 
Theorem \ref{thm:mai3}, \(r=2\), \(f_1(X):=\gamma(\rho;X)\), and \(f_2(X):=g(X)\). 
The first assumption is clear. 
The fourth assumption 
follows from the third assumption on \(\mathcal{F}\). \par
In order to check the third assumption, it suffices to show that 
\begin{align}
\frac{1}{\rho}<\frac1{A}
\label{eqn5:12}
\end{align}
by (\ref{eqn:tra4}) and the second assumption on \(\mathcal{F}\). 
We may assume that \(A\geq 4\) by (\ref{eqn:mai2}). 
Using 
\begin{align*}
\log \left(1-\frac1{A}\right)^A
&=-A\sum_{n=1}^{\infty}\frac1{n}A^{-n}\\ &>-A\sum_{n=1}^{\infty}A^{-n}
=-1-\frac{1}{A-1},
\end{align*}
we get by \(A\geq 4\) that 
\begin{align*}
\left(1-\frac1{A}\right)^A& >\exp\left(-1-\frac{1}{A-1}\right)\nonumber\\
&\geq \exp\left(-\frac43\right)>\frac14\geq \frac1{A}. 
\end{align*}
Hence, we obtain 
\[G_A\left(\frac1{A}\right)=\left(1-\frac1{A}\right)^A-\frac1{A}>0,\]
which implies (\ref{eqn5:12}) by (\ref{eqn:mai2}) 
and the second statement of Lemma \ref{lem:mai1}. 
In what follows, we check the second assumption of Theorem \ref{thm:cri1}. 
The following lemma was inspired by the results of Daniel \cite{dan}. 
\begin{lem}
Let \(k\) be a positive integer. Then 
\begin{align}
R-\theta\bigl(R;kS(f_1)\bigr)=O\left(R^{(1-1/{\rho})^k}\right)
\label{eqn5:14}
\end{align}
for any \(R\geq 1\), where the implied constant in the symbol \(O\) does not depend on \(R\), but 
on \(k\). 
\label{lem5:3}
\end{lem}
\begin{proof}
First we consider the case of \(k=1\). Using the mean value theorem, we see 
that 
\begin{align}
\lfloor (m+1)^{\rho}\rfloor - \lfloor m^{\rho}\rfloor
&= (m+1)^{\rho}-m^{\rho}+O(1)\nonumber\\
&=O\left(m^{{\rho}-1}\right)=O\left(\left\lfloor m^{\rho}\right\rfloor^{1-1/{\rho}}\right)
\label{eqn5:15}
\end{align}
for any positive integer \(m\). For any sufficiently large \(R\), take a positive integer \(m\) with 
\[\lfloor m^{\rho}\rfloor < R \leq \lfloor (m+1)^{\rho}\rfloor\]
Then we get 
\[R-\theta\bigl(R;S(f_1)\bigr)\leq \lfloor (m+1)^{\rho}\rfloor - \lfloor m^{\rho}\rfloor
=O\left(R^{1-1/{\rho}}\right)\]
by (\ref{eqn5:15}). \par
Next, we assume that (\ref{eqn5:14}) holds for a positive integer \(k\). Let 
\[R_0:=R-\theta\bigl(R;kS(f_1)\bigr)\in\mathbb{Z}^{+}.\]
The inductive hypothesis implies that 
\begin{align}
R_0=O\left(R^{(1-1/{\rho})^k}\right).
\label{eqn5:16}
\end{align}
Set 
\[\eta:=\theta\bigl(R;kS(f_1)\bigr)+\theta\bigl(R_0;S(f_1)\bigr).\]
Then we have \(\eta\in (k+1)S(f_1)\) and 
\begin{align}
R-\eta=R_0-\theta\bigl(R_0;S(f_1)\bigr)>0.
\label{eqn5:17}
\end{align}
Thus, 
\begin{align}
\theta\bigl(R;(k+1)S(f_1)\bigr)\geq \eta.
\label{eqn5:18}
\end{align}
Combining (\ref{eqn5:17}) and (\ref{eqn5:18}), we obtain 
\begin{align*}
 R- \theta\bigl(R;(k+1)S(f_1)\bigr)&\leq R-\eta\\
&=R_0-\theta\bigl(R_0;S(f_1)\bigr)
\end{align*}
Consequently, using (\ref{eqn5:14}) with 
\(k=1\) and \(R=R_0\), we deduce that 
\begin{align*}
0&< R- \theta\bigl(R;(k+1)S(f_1)\bigr)\\
&=O\left(R_0^{1-1/{\rho}}\right)=O\left(R^{(1-1/{\rho})^{k+1}}\right)
\end{align*}
by (\ref{eqn5:16}). 
\end{proof}
Using Lemma \ref{lem5:3} with \(k=1+k_1\), we get 
\begin{align*}
\log_R F_1(R)&:=\log_R\Bigl(R-\theta\bigl(R;(1+k_1)S(f_1)\bigr)\Bigr)\\
&\leq \left(1-\frac1{\rho}\right)^{1+k_1}+o(1)
\end{align*}
as \(R\) tends to infinity. Moreover, using (\ref{eqn:tra4}) and the second assumption on \(\mathcal{F}\), 
we see 
\begin{align*}
\log_R F_2(R):=\log_R\left(
\frac{R}{\prod_{i=1}^2\lambda(S(f_i);R)^{k_i}}
\right)
=1-\frac{k_1}{\rho}+o(1). 
\end{align*}
Thus, we obtain 
\[\log_R F_1(R)-\log_R F_2(R)\leq G_{1+k_1}\left(\frac1{\rho}\right)+o(1)\]
as \(R\) tends to infinity. 
For the proof of (\ref{ppp}), 
it suffices to show that 
\begin{align}
G_{1+k_1}\left(\frac1{\rho}\right)<0.
\label{eqn5:19}
\end{align}
In fact, (\ref{eqn5:19}) implies that there exists a positive constant \(c\) satisfying 
\[F_1(R)< R^{-c} F_2(R)\]
for any sufficiently large \(R\). \par
If \(k_1=0\) or \(k_1=1\), then (\ref{eqn5:19}) is clear by 
\(G_1(X)=-X\) and \(G_2(X)=-X(1-X)\). If \(k_1=2\), then we have \(G_3(X)=-X(1-3X+X^2)\) and 
\(\sigma_3=(3-\sqrt{5})/2\). 
By (\ref{eqn5:12}) and (\ref{eqn:cri1}), we get 
\[\frac1{\rho}<\frac1{A}\leq \frac1{1+k_1}=\frac13<\sigma_3,\]
which implies (\ref{eqn5:19}) by the second statement of Lemma \ref{lem:mai1}. Finally, suppose that 
\(k_1\geq 3\). Using (\ref{eqn:mai2}), (\ref{eqn:cri1}), and the third statement of Lemma \ref{lem:mai1}, 
we obtain 
\[\frac1{\rho}<\sigma_A\leq \sigma_{1+k_1}, \]
which means (\ref{eqn5:19}). Therefore, we proved Theorem \ref{thm:mai3}. 
\end{proof}
\begin{proof}[Proof of Corollary \ref{cor:mai3}]
The first statement of Corollary \ref{cor:mai3} follows from Theorem \ref{thm:mai3} by 
\[\sum_{m=0}^{\infty}X^{\lfloor y^m\rfloor}\in\mathcal{F}.\]
The second statement of the corollary is similarly verified by \(\xi(y,z;X)\in \mathcal{F}\). 
In fact, the second assumption on \(\mathcal{F}\) follows from the fact that, for any 
real number \(M\), 
\[\lim_{R\to\infty}\frac{\varphi(y,z;R)}{R^M}=\infty.\]
Moreover, in the same way as the proof of (\ref{eqn5:6}), 
we can show that 
\[\lim_{R\to\infty}\frac{\varphi(y,z;R+1)}{\varphi(y,z;R)}=1.\]
\end{proof}
\begin{proof}[Proof of Corollary \ref{cor:mai4}]
By Theorem \ref{thm:mai3} and the second statement of Lemma \ref{lem:mai1}, it suffices to 
show that \((\varepsilon+1/2)A^2>\sigma_A^{-1}\), namely, 
\[0> G_A\left(\left(\frac12+\varepsilon\right)^{-1}A^{-2}\right)\]
for any sufficiently large \(A\), depending only on \(\varepsilon>0\). 
We now fix an arbitrary positive real number \(\varepsilon\). In the proof of Corollary \ref{cor:mai4}, 
the implied constant in the symbol \(O\) does not depend on \(A\), but on \(\varepsilon\). 
Observe that 
\begin{align*}
&\log \left(1-\left(\frac12+\varepsilon\right)^{-1}A^{-2}\right)^A\\
&=A\left(-\left(\frac12+\varepsilon\right)^{-1}A^{-2}+O\left(A^{-4}\right)\right)\\
&=-\left(\frac12+\varepsilon\right)^{-1}A^{-1}+O\left(A^{-3}\right)
\end{align*}
and that 
\begin{align*}
&\left(1-\left(\frac12+\varepsilon\right)^{-1}A^{-2}\right)^A\\
&= \exp\left(-\left(\frac12+\varepsilon\right)^{-1}A^{-1}+O\left(A^{-3}\right)\right)\\
&=
1-\left(\frac12+\varepsilon\right)^{-1}A^{-1}+\frac12
\left(\frac12+\varepsilon\right)^{-2}A^{-2}
+O\left(A^{-3}\right).
\end{align*}
Thus, we get 
\begin{align*}
& G_A\left(\left(\frac12+\varepsilon\right)^{-1}A^{-2}\right)\\
&= \left(1-\left(\frac12+\varepsilon\right)^{-1}A^{-2}\right)^A\\
&\hspace{20mm} -1+\left(\frac12+\varepsilon\right)^{-1}A^{-1}
-\left(\frac12+\varepsilon\right)^{-1}A^{-2}\\
&=-\varepsilon \left(\frac12 +\varepsilon\right)^{-2} A^{-2}
+O\left(A^{-3}\right)<0
\end{align*}
for any sufficiently large \(A\), depending only on \(\varepsilon\). 
\end{proof}
\section{Proof of our criteria}
\subsection{Proof of Theorem \ref{thm:cri2}}
We prove Theorem \ref{thm:cri2} by Theorem \ref{thm:cri1}, 
showing that 
\[f_1(X):=1+\sum_{m=m_0}^{\infty} X^{\lfloor a(m)\rfloor}, \ 
f_2(X):=1+\sum_{m=m_0}^{\infty} X^{\lfloor u(m)\rfloor}\]
satisfy the assumptions of Theorem \ref{thm:cri1}, 
where \(r=2\) and \(A\) is any fixed positive integer. 
The first assumption is trivial. The fourth assumption 
of Theorem \ref{thm:cri1} follows from the first assumption on \(u(R)\). 
Using (\ref{qqq}) and 
the second assumption on \(u(R)\), we get the following: For any positive real number \(\varepsilon\), 
\begin{align}
\lambda\bigl(S(f_1);R\big)&\sim b(R)=o\left( R^{\varepsilon}\right)
\label{eqn6:1},\\
\lambda\bigl(S(f_2);R\big)&\sim v(R)=o\left( b(R)^{\varepsilon}\right)=
o\left( \lambda\bigl(S(f_1);R\big)^{\varepsilon}\right)
\label{eqn6:2}
\end{align}
as \(R\) tends to infinity, which implies that the third assumption on Theorem \ref{thm:cri1} holds. \par
In what follows, we check the second assumption. 
In the same way as the proof of Lemma \ref{lem5:3}, we show 
the following: 
\begin{lem}
Let \(k\) be a positive integer and \(\varepsilon\) a positive real number. 
Then we have 
\begin{align}
R-\theta\bigl(R;kS(f_1)\bigr)\ll \frac{R}{b(R)^{k-\varepsilon}}
\label{eqn6:3}
\end{align}
for any \(R\geq 1\), where the implied constant in the symbol \(\ll\) does not depend on \(R\), 
but on \(k\) and \(\varepsilon\).  
\label{lem6:1}
\end{lem}
\begin{proof}
It suffices to show for each \(k\geq 1\) that, for any \(\varepsilon>0\), (\ref{eqn6:3}) holds for any sufficiently large \(R\), depending on \(k\) and \(\varepsilon\). 
We prove the lemma by induction on \(k\). \par
We first consider the case of \(k=1\). We may assume that \(\varepsilon<1\). 
By the second assumption on 
\(a(m)\) and the mean value theorem, we get for any sufficiently large \(m\) that 
\[a(m)\leq a(m+1)\leq 2 a(m)\]
and that 
there exists a real number \(\rho\) with  \(0<\rho<1\) satisfying 
\begin{align}
a(m+1)-a(m)&=a'(m+\rho)
<\frac{a(m+\rho)}{(m+\rho)^{1-\varepsilon}}\nonumber\\
&\leq \frac{a(m+1)}{m^{1-\varepsilon}}\ll \frac{a(m)}{(m+1)^{1-\varepsilon}}.
\label{eqn6:4}
\end{align}
For any sufficiently large \(R\), there exists an integer \(m\geq m_0\) such that 
\[\lfloor a(m)\rfloor < R\leq \lfloor a(m+1)\rfloor .\]
By (\ref{eqn6:4}), we obtain 
\begin{align*}
R-\theta\bigl(R;S(f_1)\bigr)
& = R- \lfloor a(m)\rfloor\leq a(m+1)-a(m)+1\\
& \ll \frac{a(m)}{(m+1)^{1-\varepsilon}} \ll \frac{R}{b(a(m+1))^{1-\varepsilon}}
\leq \frac{R}{b(R)^{1-\varepsilon}},
\end{align*}
which implies (\ref{eqn6:3}) in the case of \(k=1\). \par
Next we assume that (\ref{eqn6:3}) holds for a fixed positive integer \(k\) and an arbitrary positive 
real number \(\varepsilon\). In what follows, we verify (\ref{eqn6:3}) for \(k+1\) 
with fixed \(\varepsilon<1\). 
Put 
\[R_0:=R-\theta\bigl(R;kS(f_1)\bigr). \]
It suffices to consider the case of 
\begin{align}
R_0\geq \frac{R}{b(R)^{k+1}}.
\label{eqn6:5}
\end{align}
In fact, suppose that (\ref{eqn6:5}) does not hold. Since \(0\in S(f_1)\) by the definition of \(f_1(X)\), 
we have 
\[\theta\bigl(R;kS(f_1)\bigr)\in k S(f_1)\subset (k+1) S(f_1)\]
by Remark \ref{rem:cri1}. Thus, we get 
\[R-\theta\bigl(R;(k+1)S(f_1)\bigr)\leq R_0 <\frac{R}{b(R)^{k+1}},\]
which implies (\ref{eqn6:3}). \par
In what follows, we assume that (\ref{eqn6:5}) is satisfied. In particular, 
applying (\ref{eqn6:1}) to (\ref{eqn6:5}), we see 
\begin{align}
R_0\geq R^{1-\varepsilon/4}
\label{eqn6:6}
\end{align}
for any sufficiently large \(R\). Moreover, the inductive hypothesis implies that 
\begin{align}
R_0\ll \frac{R}{b(R)^{k-\varepsilon/2}}.
\label{eqn6:7}
\end{align}
In the same way as the proof of Lemma \ref{lem5:3}, putting 
\[\eta:= \theta\bigl(R;kS(f_1)\bigr)+\theta\bigl(R_0;S(f_1)\bigr)\in (k+1) S(f_1),\]
we see that 
\begin{align*}
R-\theta\bigl(R;(k+1)S(f_1)\bigr)& \leq R-\eta\\
&=R_0-\theta\bigl(R_0;S(f_1)\bigr)\ll \frac{R_0}{b(R_0)^{1-\varepsilon/4}},
\end{align*}
where for the last inequality we apply (\ref{eqn6:3}) with \(k=1\). 
By (\ref{eqn6:6}) and (\ref{eqn6:7}), we obtain 
\begin{align}
R-\theta\bigl(R;(k+1)S(f_1)\bigr)\ll \frac{R}{b(R)^{k-\varepsilon/2}b(R^{1-\varepsilon/4})^{1-\varepsilon/4}}.
\label{eqn6:8}
\end{align}
Using the assumption that \((\log a(x))/(\log x)\) is ultimately increasing with 
\[x=b(R)> x'= b(R^{1-\varepsilon/4}),\]
we get 
\begin{align*}
\frac{\log R}{\log b(R)}&=\frac{\log a(x)}{\log x}\geq \frac{\log a(x')}{\log x'}\\
&=\left(1-\frac{\varepsilon}{4}\right) \frac{\log R}{\log b(R^{1-\varepsilon/4})}.
\end{align*}
Consequently, 
\[b\left( R^{1-\varepsilon/4}\right)\geq b(R)^{1-\varepsilon/4},\]
and so 
\begin{align}
\frac{1}{b(R^{1-\varepsilon/4})^{1-\varepsilon/4}}\leq \frac{1}{b(R)^{(1-\varepsilon/4)^2}}
\leq \frac{1}{b(R)^{1-\varepsilon/2}}
\label{eqn6:9}
\end{align}
by \((1-\varepsilon/4)^2\geq 1-\varepsilon/2\). Combining (\ref{eqn6:8}) and (\ref{eqn6:9}), we 
deduce that 
\begin{align*}
(0<)R-\theta\bigl(R;(k+1)S(f_1)\bigr)\ll \frac{R}{b(R)^{k+1-\varepsilon}},
\end{align*}
which implies (\ref{eqn6:3}). 
\end{proof}
Let \(k_1,k_2\) be nonnegative integers. Applying Lemma \ref{lem6:1} with 
\(k=1+k_1\) and \(\varepsilon=1/2\), we deduce 
by (\ref{eqn6:2}) that 
\begin{align*}
R-\theta(R;(1+k_1) S(f_1)) & \ll \frac{R}{b(R)^{k_1+1/2}}\\
&=o\left(\frac{R}{b(R)^{k_1} v(R)^{k_2}}\right)=o\left(\frac{R}{\prod_{h=1}^2 \lambda(S(f_h);R)^{k_h}}\right)
\end{align*}
as \(R\) tends to infinity. 
Finally, we proved Theorem \ref{thm:cri2}. 
\subsection{Proof of Theorem \ref{thm:cri1}}
Put 
\begin{align*}
\overline{f_i}(X):=\left\{
\begin{array}{cc}
f_i(X) & \mbox{if }f_i(0)\ne 0, \\
1+f_i(X) & \mbox{if }f_i(0)=0.
\end{array}
\right.
\end{align*}
Then \(\overline{f_1}(X),\ldots,\overline{f_r}(X)\) satisfy the assumptions of Theorem \ref{thm:cri1}. 
The first and fourth assumptions are easily checked. 
Moreover, the second and the third assumptions are also seen by 
\begin{align*}
&\theta\left(R;\sum_{h=1}^{r-2} 
k_h S \hspace{-0.7mm}\left( \hspace{0.3mm}\overline{f_h}\hspace{0.3mm}\right)+
(1+k_{r-1})S\hspace{-0.7mm}\left(\hspace{0.3mm}\overline{f_{r-1}}\hspace{0.3mm}\right)\right)\\
&\hspace{10mm}\geq \theta\left(R;\sum_{h=1}^{r-2} k_h S(f_h)+(1+k_{r-1})S(f_{r-1})\right)
\end{align*}
and, for \(h=1,\ldots,r\), 
\[\lambda\left(S\left(\hspace{0.3mm}\overline{f_h}\hspace{0.3mm}\right);R\right) \sim \lambda(S(f_h);R)\]
as \(R\) tends to infinity. For the proof of Theorem \ref{thm:cri1}, 
it suffices to show that 
\[\left\{\left. \overline{f_1}(\beta^{-1})^{k_1}
\overline{f_2}(\beta^{-1})^{k_2}\cdots \overline{f_r}(\beta^{-1})^{k_r}
\ \right| \ k_1,k_2,\ldots, k_r\in\mathbb{N}, k_1\leq A\right\}\]
is linearly independent over \(\mathbb{Q}(\beta)\). In particular, rewriting \(\overline{f_i}(X)\) by 
\(f_i(X)\) for \(i=1,\ldots,r\), we may assume that \(f_i(0)\ne 0\) for any \(i=1,\ldots, r\). \par
For simplicity, put, for \(i=1,\ldots,r\), 
\[\xi_i:=f_i(\beta^{-1}), S_i:=S(f_i), \lambda_i(R):=\lambda\bigl(S(f_i);R\bigr).\]
Using Corollary \ref{cor:tra1} and the third assumption of Theorem \ref{thm:cri2}, we see that 
\[[\mathbb{Q}(\xi_1,\beta):\mathbb{Q}(\beta)]\geq A+1\]
and that \(\xi_2,\ldots,\xi_r\) are transcendental. \par
We introduce notation for the proof of Theorem \ref{thm:cri1}. 
For any nonempty subset \(\mathcal{A}\) of \(\mathbb{N}\) and any positive integer \(k\), let 
\(\mathcal{A}^k\) denote the \(n\)-fold Cartesian product. For convenience, set 
\[\mathcal{A}^0:=\{0\}.\]
Let \(k\in \mathbb{N}\) and \(\bp=(p_1,\ldots,p_k)\in \mathbb{N}^k\). We put 
\begin{align*}
|\bp|:=\left\{
\begin{array}{cc}
0 & (k=0), \\
p_1+\cdots+p_k & (k\geq 1)
\end{array}
\right.
\end{align*}
and, for \(i=1,\ldots, r\), 
\begin{align*}
t_i (\bp):=
\left\{
\begin{array}{cc}
1 & (k=0), \\
t_i(p_1)\cdots t_i(p_k) & (k\geq 1).
\end{array}
\right.
\end{align*}
Moreover, for any \(\bk=(k_1,\ldots,k_r)\in \mathbb{N}^r\), let 
\[\underline{X}^{\bk}=\prod_{i=1}^{r} X_i^{k_i}, \ 
\underline{\xi}^{\bk}:=\prod_{i=1}^r \xi_i^{k_i}, \ \underline{\lambda}(N)^{\bk}:=
\prod_{i=1}^r \lambda_i (N)^{k_i}. \]
We calculate \(\underline{\xi}^{\bk}\) in the same way as the proof of Theorem 2.1 in \cite{kan1}. 
The method was inspired by the proof of Theorem 7.1 in \cite{bai}. 
Let \(\bk\in\mathbb{N}^r\backslash\{(0,\ldots,0)\}\). Then we have 
\begin{align}
\underline{\xi}^{\bk}
& =
\prod_{i=1}^r \left(\sum_{m_i\in S_i} t_i(m_i) \beta^{-m_i}\right)^{k_i}
\nonumber\\
&=\prod_{i=1}^r\sum_{\bm_i\in S_i^{k_i}} t_i(\bm_i) \beta^{-|\bm_i|}
=:
\sum_{m=0}^{\infty} \beta^{-m} \rho(\bk;m),
\label{pr1}
\end{align}
where 
\begin{align*}
\rho(\bk;m)=\sum_{\bm_1\in S_1^{k_1},\ldots,\bm_r\in S_r^{k_r}\atop 
|\bm_1|+\cdots+|\bm_r|=m} t_1(\bm_1)\cdots t_r(\bm_r)\in\mathbb{N}. 
\end{align*}
Note that \(\rho(\bk;m)\) is positive if and only if 
\[m\in \sum_{h=1}^r k_h S_h.\]
We see that 
\begin{align}
\rho(\bk;m)\leq 
\sum_{\bm_1\in S_1^{k_1},\ldots,\bm_r\in S_r^{k_r}\atop 
|\bm_1|+\cdots+|\bm_r|=m} C_7^{|\bk|} \leq C_7^{|\bk|} (1+m)^{|\bk|}.
\label{pr2}
\end{align}
We give an analogue of Lemma 4.1 in \cite{kan1}. 
\begin{lem}
Let \(\bk\in \mathbb{N}^r\backslash \{(0,\ldots,0)\}\) and let \(N\in\mathbb{Z}^{+}\). 
Then we have 
\begin{align}
\sum_{m=0}^{N-1}\rho(\bk;m)\leq C_7^{|\bk|} \underline{\lambda}(N)^{\bk}
\label{pr3}
\end{align}
and 
\begin{align}
\mbox{Card}\hspace{0.6mm} \{m\in\mathbb{N}\mid 
m< N, \rho(\bk;m)>0\}\leq C_7^{|\bk|} \underline{\lambda}(N)^{\bk}.
\label{pr4}
\end{align}
\label{prlem1}
\end{lem}
\begin{proof}
We see that (\ref{pr4}) follows from (\ref{pr3}) because \(\rho(\bk;m)\in\mathbb{N}\) for any 
\(m\). Put \(S(i;N):=S_i\cap [0,N)\) for \(i=1,\ldots,r\). Then we get 
\begin{align*}
\sum_{m=0}^{N-1} \rho(\bk;m)
&=\sum_{\bm_1\in S_1^{k_1},\ldots,\bm_r\in S_r^{k_r}\atop 
|\bm_1|+\cdots+|\bm_r|< N} t_1(\bm_1)\cdots t_r(\bm_r)\\
&\leq 
C_7^{|\bk|}\sum_{\bm_1\in S(1;N)^{k_1}}\sum_{\bm_2\in S(2;N)^{k_2}}
\cdots \sum_{\bm_r\in S(r;N)^{k_r}} 1\\
&= C_7^{|\bk|} \underline{\lambda}(N)^{\bk},
\end{align*}
which implies (\ref{pr3}). 
\end{proof}
Assume that the set \(\{\underline{\xi}^{\bk}\mid \bk=(k_1,\ldots,k_r)\in \mathbb{N}^r, k_1\leq A\}\) 
is linearly independent over \(\mathbb{Q}(\beta)\). Then there exists 
\(P(X_1,\ldots,X_r)\in \mathbb{Z}[\beta][X_1,\ldots,X_r]\backslash\mathbb{Z}[\beta]\) such that 
the degree of \(P(X_1,\ldots,X_r)\) in \(X_1\) is at most \(A\) and that 
\begin{align}
P(\xi_1,\ldots,\xi_r)=0. 
\label{pr5}
\end{align}
Let \(D\) be the total degree of \(P(X_1,\ldots,X_r)\). 
Without loss of generality, we may assume that \(X_r(-1+X_r)\) divides \(P(X_1,\ldots,X_r)\) and 
that if \(r\geq 3\), then \(X_{r-1}\) divides \(P(X_1,\ldots,X_r)\). Put 
\begin{align}
P(X_1,\ldots,X_r)=:\sum_{\bk\in \Lambda} A_{\bk} \underline{X}^{\bk}, 
\label{pr6}
\end{align}
where \(\Lambda\) is a nonempty finite subset of \(\mathbb{N}^r\) and 
\(A_{\bk}\in \mathbb{Z}[\beta]\backslash \{0\}\) for any \(\bk\in \Lambda\). 
For any \(\bk=(k_1,\ldots,k_r) \in \Lambda\), we have \(k_r\geq 1\) because 
\(X_r\) divides \(P(X_1,\ldots,X_r)\). Moreover, if \(r\geq 3\), then 
\begin{align}
k_{r-1}\geq 1
\label{pr7}
\end{align}
because \(X_{r-1}\) divides \(P(X_1,\ldots,X_r)\). \par
The lexicographic order \(\succ\) on \(\mathbb{N}^r\) is defined as follows: 
Let \(\bk=(k_1,\ldots,k_r)\) and \(\bk'=(k_1',\ldots,k_r')\) be distinct 
elements of \(\mathbb{N}^r\). 
Put \(l:=\min \{i\mid 1\leq i\leq r, k_i\ne k_i'\}\). Then \(\bk\succ \bk'\) if and only if 
\(k_l>k_l'\). The third assumption of Theorem \ref{thm:cri1} implies that if \(\bk\succ \bk'\), then 
\begin{align}
\underline{\lambda}(N)^{\bk'}=o\left(\underline{\lambda}(N)^{\bk}\right)
\label{pr8}
\end{align}
as \(N\) tends to infinity. \par
Let \(\bg=(g_1,\ldots,g_r)\) be the greatest element of \(\Lambda\) with respect to \(\succ\). 
Without loss of generality, we may assume that 
\begin{align}
A_{\bg}\geq 1.
\label{pr9}
\end{align}
We see that  
\begin{align}
g_{r-1}\geq 1.
\label{pr10}
\end{align}
In fact, (\ref{pr10}) follows from (\ref{pr7}) if \(r\geq 3\). Suppose that \(r=2\). Then \(g_1\) is the 
degree of \(P(X_1,X_2)\) in \(X_1\). Thus, \(g_1\) is positive because \(\xi_2\) is transcendental. \par
Putting 
\[\Lambda_1:=\{\bk=(k_1,\ldots,k_{r-1},k_r)\mid k_1=g_1,\ldots,k_{r-1}=g_{r-1},k_r<g_r\}\]
and 
\[\Lambda_2:=\{\bk=(k_1,\ldots,k_{r-1},k_r)\mid k_i< g_i \mbox{ for some }i\leq r-1\},\]
we see \(\Lambda=\{\bg\}\cup \Lambda_1\cup\Lambda_2\). Using the fact that \(\xi_r\) is 
transcendental and that \(-1+X_r\) divides \(P(X_1,\ldots,X_r)\), we obtain the following lemma, 
applying the same method as the proof of 
Lemma 4.3 in \cite{kan1} with \(F(X_{r-1},X_r)=1\): 
\begin{lem}
\(\Lambda_1\) and \(\Lambda_2\) are not empty. 
\label{prlem2}
\end{lem}
Set 
\[\be=(g_1,\ldots,g_{r-2},-1+g_{r-1},1+D).\]
Recall that the degree \(g_1\) of \(P(X_1,\ldots,X_r)\) in \(X_1\) is at most \(A\). Thus, we 
can apply the second assumption of Theorem \ref{thm:cri1} with \(\bk=(k_1,\ldots,k_r)=\be\). 
In fact, we see 
\begin{align*}
k_1=\left\{
\begin{array}{cc}
-1+g_1 & (\mbox{if }r=2),\\
g_1 & (\mbox{if }r\geq 3).
\end{array}
\right.
\end{align*}
Hence, there exits a positive constant \(C_{12}\) satisfying the following: 
For any integer \(R\) with \(R\geq C_{12}\), we have 
\begin{align}
\lambda_r(R)\geq 5
\label{pr11}
\end{align}
and 
\begin{align}
R-\theta\left(\sum_{h=1}^{r-1} g_h S_h;R\right)< \frac{R}{\underline{\lambda}(R)^{\be}}.
\label{pr12}
\end{align}
In what follows, we set 
\[\theta(R):=\theta\left(\sum_{h=1}^{r-1} g_h S_h;R\right)\]
for simplicity. 
Using (\ref{pr11}) and (\ref{pr12}), we obtain the following lemma in the same way as the proof of 
Lemma 4.4 in \cite{kan1}: 
\begin{lem}
Let \(M,E\) be real numbers with 
\[M\geq C_{12}, E\geq \frac{4M}{\underline{\lambda}{(M)}^\be}.\]
Then 
\[M+\frac12E<\theta(M+E).\]
\label{prlem3}
\end{lem}
Using \(k_1\leq A\) and the third assumption of Theorem \ref{thm:cri1}, we get 
\[\lim_{R\to\infty}\frac{R}{\underline{\lambda}(R)^{\be}}=\infty.\]
Thus, the set 
\[\Xi:=\left\{
N\in\mathbb{N}\left|
\frac{N}{\underline{\lambda}(N)^{\be}}\geq \frac{n}{\underline{\lambda}(n)^{\be}}
\mbox{ for any }n\leq N\right.\right\}
\]
is infinite. We now verify for any \(\bk=(k_1,\ldots,k_r) \in\Lambda_2\) that 
\begin{align}
\underline{\lambda}(N)^{\bk}=o\left(\underline{\lambda}(N)^{\be}\right)
\label{pr13}
\end{align}
as \(N\) tends to infinity. For the proof of (\ref{pr13}), it suffices to check 
\begin{align}
\be\succ \bk
\label{pr14}
\end{align}
by (\ref{pr8}). 
If \(g_i>k_i\) for some \(i\leq r-2\), then (\ref{pr14}) holds. Suppose that 
\(g_i=k_i\) for any \(i\leq r-2\). Then we get \(-1+g_{r-1}\geq k_{r-1}\) and \(1+D>k_r\) by 
\(\bk\in \Lambda_2\), which implies (\ref{pr14}). \par
Combining (\ref{pr5}), (\ref{pr6}), and (\ref{pr1}), we get 
\[0=\sum_{\bk\in \Lambda} A_{\bk} \underline{\xi}^{\bk}
=\sum_{\bk\in \Lambda} A_{\bk} \sum_{m=0}^{\infty} \rho(\bk;m)\beta^{-m}.\]
For an arbitrary nonnegative integer \(R\), multiplying \(\beta^R\) to the both-hand sides of the equality above, 
we obtain 
\[0=\sum_{\bk\in \Lambda} A_{\bk} \sum_{m=-R}^{\infty} \rho(\bk;m+R)\beta^{-m}.\]
Putting 
\begin{align}
Y_R&:=
\sum_{\bk\in \Lambda} A_{\bk} \sum_{m=1}^{\infty} \rho(\bk;m+R)\beta^{-m}\nonumber \\
&=
-
\sum_{\bk\in \Lambda} A_{\bk} \sum_{m=-R}^{0} \rho(\bk;m+R)\beta^{-m},
\label{pr15}
\end{align}
we see that \(Y_R\) is an algebraic integer because \(\beta\) is a Pisot or Salem number. 
\begin{lem}
There exist positive integers \(C_{13}\) and \(C_{14}\) satisfying the following: 
For any integer \(R\) with \(R\geq C_{14}\), we have 
\[Y_R=0\mbox{, or }|Y_R|\geq R^{-C_{13}}.\]
\label{prlem4}
\end{lem}
\begin{proof}
Let \(d\) be the degree of \(\beta\) and let \(\sigma_1,\sigma_2,\ldots,\sigma_d\) be the 
conjugate embeddings of \(\mathbb{Q}(\beta)\) into \(\mathbb{C}\) such that 
\(\sigma_1(\gamma)=\gamma\) for any \(\gamma \in \mathbb{Q}(\beta)\). 
Set 
\[C_{15}:=\max\{|\sigma_i(A_{\bk})|\mid i=1,\ldots,d, \bk\in\Lambda\}.\]
Let \(2\leq i\leq d\). Using (\ref{pr15}) and (\ref{pr2}), and \(|\beta_i|\leq 1\), we get 
\begin{align*}
|\sigma(Y_R)|&=
\left|
\sum_{\bk\in\Lambda}\sigma_i(A_{\bk})\sum_{n=0}^R\rho(\bk;-n+R)\sigma_i(\beta)^n
\right|\\
&\leq\sum_{\bk\in \Lambda}C_{15}\sum_{n=0}^R C_7^{D}(1+R)^D\ll (R+1)^{D+1}.
\end{align*}
In particular, if \(R\gg 1\), then 
\[|\sigma(Y_R)|\leq R^{D+2}.\]
Hence, if \(Y_R\ne 0\), then we obtain 
\[1\leq |Y_R|\prod_{i=2}^d |\sigma(Y_R)|\leq |Y_R| R^{(D+2)(d-1)}.\]
\end{proof}
In the case of \(\beta=2\) and \(r=1\), Bailey, Borwein, Crandall, and Pomerance estimated 
the numbers \(\widetilde{y_N}\) of positive \(Y_R\) with \(R< N\)
in order to give lower bounds for the nonzero digits 
in binary expansions (Theorem 7.1 in \cite{bai}). 
Moreover, if \(\beta=b>1\) is a rational integer and \(r\geq 2\), then \(\widetilde{y_N}\) is applied to 
prove a criterion for algebraic independence (Theorem 2.1 in \cite{kan1}). \par
Now, we put, for \(N\in \mathbb{Z}^{+}\), 
\[y_N:=\mbox{Card}\left\{R\in\mathbb{N}\ \left| \ R< N, Y_R\geq \frac1{\beta}\right.\right\}.\]
In the case where \(\beta\) is a Pisot or Salem number and \(r=1\), then 
\(y_N\) is estimated to give lower bounds for the numbers of nonzero digits in \(\beta\)-expansions 
(Theorem 2.2 in \cite{kan3}). In what follows, we calculate upper and lower bounds for \(y_N\), which 
gives contradiction. First, we estimate upper bounds for \(y_N\) in Lemma \ref{prlem5}. Next, we 
give lower bounds for \(y_N\) in Lemma \ref{prlem10}, estimating upper bounds for \(R-\theta(R;\Omega)\) 
in Lemma \ref{prlem9}, where 
\begin{align}
\Omega=\left\{R\in\mathbb{N}\left|Y_R\geq \frac1{\beta}\right.\right\}.
\label{Omega}
\end{align}
In what follows, we assume that 
\(N\) is a sufficiently large integer satisfying 
\begin{align}
\left(1+\frac1{N}\right)^D<\frac{\beta+1}2. \label{pr16}
\end{align}
\begin{lem}We have 
\[y_N=o\left(N^{1-\delta/2}\right)\]
as \(N\) tends to infinity. 
\label{prlem5}
\end{lem}
\begin{proof}
Put \[K:=\lceil (1+D)\log_{\beta} N\rceil.\]
Then we see 
\begin{align*}
y_N&\leq K+y_{N-K}
=K+\sum_{0\leq R< N-K\atop Y_R\geq 1/\beta} 1\\
&\leq K+\beta\sum_{R=0}^{N-K-1}|Y_R|
\end{align*}
and 
\begin{align*}
\sum_{R=0}^{N-K-1}|Y_R|&\leq 
\sum_{R=0}^{N-K-1}\sum_{\bk\in \Lambda}\sum_{m=1}^{\infty} |A_{\bk}|\beta^{-m}\rho(\bk;m+R)\\
&= \sum_{\bk\in\Lambda} |A_{\bk}|Y(\bk;N),
\end{align*}
where 
\[Y(\bk;N)=\sum_{R=0}^{N-K-1}\sum_{m=1}^{\infty} \beta^{-m}\rho(\bk;m+R)\]
for \(\bk\in \Lambda\). For the proof of Lemma \ref{prlem5}, it suffices to show for any 
\(\bk=(k_1,k_2,\ldots,k_r)\in\Lambda\) that 
\begin{align}
Y(\bk;N)=o\left(N^{1-\delta/2}\right)
\label{pr17}
\end{align}
as \(N\) tends to infinity. Observe that 
\begin{align}
0\leq Y(\bk;N)&=\sum_{m=1}^{K}\sum_{R=0}^{N-K-1} \beta^{-m}\rho(\bk;m+R)\nonumber\\
&\hspace{10mm}
+\sum_{m=K+1}^{\infty}\sum_{R=0}^{N-K-1} \beta^{-m}\rho(\bk;m+R)\nonumber\\
&=:S^{(1)}(\bk;N)+S^{(2)}(\bk;N).
\label{pr18}
\end{align}
Using (\ref{pr3}), we get 
\begin{align*}
S^{(1)}(\bk;N)
&\leq 
\sum_{m=1}^{K}\beta^{-m}\sum_{R=0}^{N-1} \rho(\bk;R)
\leq 
\sum_{m=1}^{\infty}\beta^{-m}\sum_{R=0}^{N-1} \rho(\bk;R)\\
&\leq \sum_{m=1}^{\infty}\beta^{-m}C_7^D \underline{\lambda}(N)^{\bk}
\ll \underline{\lambda}(N)^{\bk}.
\end{align*}
Thus, the third assumption of Theorem \ref{thm:cri1} implies that 
\begin{align}
S^{(1)}(\bk;N)\ll \lambda_1(N)^A\prod_{i=2}^r\lambda_i(N)^{k_i}
=o\left(N^{1-\delta/2}\right).
\label{pr19}
\end{align}
Using (\ref{pr2}), we see 
\begin{align*}
S^{(2)}(\bk;N)&\leq \sum_{m=K+1}^{\infty}\beta^{-m}\sum_{R=0}^{N-K-1}C_7^D(m+R+1)^D\\
&\ll \sum_{m=K+1}^{\infty}\beta^{-m} N(m+N)^D.
\end{align*}
Note for any \(m\in\mathbb{N}\) that 
\[\left(\frac{m+1+N}{m+N}\right)^D\leq \left(1+\frac1N\right)^D< \frac{\beta+1}2\]
by (\ref{pr16}). Hence, we obtain 
\begin{align}
S^{(2)}(\bk;N)
&\ll
\beta^{-K-1}N(K+1+N)^D\sum_{m=0}^{\infty}\beta^{-m}\left(\frac{\beta+1}2\right)^m\nonumber\\
&\ll \beta^{-K-1} N^{D+1}\leq 1.
\label{pr20}
\end{align}
Hence, combining (\ref{pr18}), (\ref{pr19}), and (\ref{pr20}), we deduce (\ref{pr17}). 
\end{proof}
In what follows, we estimate lower bounds for \(y_N\) in the case where \(N\in \Xi\) is sufficiently large. 
Recall that \(\Lambda_2\) is not empty by Lemma \ref{prlem2} and that \(0\in S_i\) 
for \(i=1,\ldots,r\). 
In particular, for any \(\bk\in \Lambda\), we have 
\(\rho(\bk;0)>0\). Put 
\begin{align*}
&\{T\in\mathbb{N}\mid T< N, \rho(\bk;T)>0\mbox{ for some }\bk\in \Lambda_2\}\\
&\hspace{30mm}=:\{0=T_1<T_2<\cdots<T_{\tau}\}.
\end{align*}
If \(N\) is sufficiently large, then (\ref{pr4}) and (\ref{pr13}) imply that 
\[\tau\leq \sum_{\bk\in\Lambda_2} C_7^{|\bk|}\underline{\lambda}(N)^{\bk}\leq 
\frac{1}{32}\underline{\lambda}(N)^{\be}. 
\]
For convenience, put \(T_{1+\tau}:=N\). Set 
\[\mathcal{J}:=\{J=J(j)\mid 1\leq j\leq \tau\},\]
where \(J(j)\) is an interval of \(\mathbb{R}\) defined by 
\(J(j)=[T_j,T_{1+j})\) for \(1\leq j\leq \tau\). \par
In what follows, we denote the length of a bounded interval \(I\) of \(\mathbb{R}\) by \(|I|\). Then we have 
\[\sum_{J\in\mathcal{J}}|J|=N. \]
Let 
\begin{align*}
\mathcal{J}_1
&:=
\left\{J\in\mathcal{J}\ \left| \ |J|\geq \frac{16 N}{\underline{\lambda}(N)^{\be}}\right.\right\},\\
\mathcal{J}_2
&:=
\{J\in \mathcal{J}_1\mid J\subset[C_{12},N)\}.
\end{align*}
In the same way as the proof of Lemma 4.7 in \cite{kan1}, we obtain the following: 
\begin{lem}
If \(N\in \Xi\) is sufficiently large, then we have 
\[
\sum_{J \in \mathcal{J}_1}|J|\geq \frac{N}2, \ 
\sum_{J \in \mathcal{J}_2}|J|\geq \frac{N}3.
\]
\label{prlem6}
\end{lem}
Recall that \(\Lambda_1\) is not empty by Lemma \ref{prlem2}. Let \(\bk_1\) be the maximal element 
of \(\Lambda_1\) with respect to \(\succ\). Set 
\begin{align*}
&\{R\in\mathbb{N}\mid R< N, \rho(\bk;R)>0\mbox{ for some }\bk\in\Lambda_1\}\\
&\hspace{30mm}=:
\{0=R_1<R_2<\cdots<R_{\mu}\}
\end{align*}
and \(R_{1+\mu}:=N\). Then (\ref{pr4}) implies that 
\[\mu\leq \sum_{\bk\in \Lambda_1} C_7^{|\bk|}\underline{\lambda}(N)^{\bk}
\leq C_{16}\underline{\lambda}(N)^{\bk_1},\]
where \(C_{16}\) is a positive constant. \par
Let 
\[\mathcal{I}:=\{I=I(i)\mid 1\leq i\leq \mu\},\]
where \(I(i)\) is an interval of \(\mathbb{R}\) defined by 
\(I(i)=[R_i,R_{i+1})\) for \(1\leq i\leq \mu\). 
Set 
\[y_N(i):=\mbox{Card}\left\{R\in I(i) \left|  Y_R\geq \frac1{\beta}\right.\right\}\]
for \(i=1,\ldots,\mu\). Observe that 
\[\sum_{I\in\mathcal{I}}|I|=N\]
and that 
\begin{align}
\sum_{i=1}^{\mu} y_N(i)=y_N.
\label{pr21}
\end{align}
Set 
\begin{align*}
\mathcal{I}_1&:=\{I\in\mathcal{I}\mid I\subset J\mbox{ for some }J\in\mathcal{J}\},\\
\mathcal{I}_2&:=\left\{I\in\mathcal{I}_1\left||I|\geq \frac1{12C_{16}}\frac{N}{\underline{\lambda}(N)^{\bk_1}}
\right.\right\}.
\end{align*}
In the same way as the proof of Lemma 4.8 in \cite{kan1}, we obtain the following: 
\begin{lem}
For any sufficiently large \(N\in \Xi\), we have 
\begin{align}
\sum_{I\in \mathcal{I}_1}|I|\geq \frac{N}6, \ \sum_{I\in \mathcal{I}_2}|I|\geq \frac{N}{12}.
\label{pr22}
\end{align}
\label{prlem7}
\end{lem}
In what follows, we assume that 
\(N\in \Xi\) satisfies 
\begin{align}
N^{\delta/2}\geq (1+C_8)C_9.
\label{pr23}
\end{align}
Let \(1\leq i\leq \mu\) with \(I(i)\in\mathcal{I}_2\) and let \(R\in (R_i,R_{i+1})\). 
We now show that 
\begin{align}
\rho(\bk;R)=0
\label{pr24}
\end{align}
for any \(\bk\in \Lambda_1\cup\Lambda_2=\Lambda\backslash\{\bg\}\). 
In fact, if \(\bk\in \Lambda_1\), then (\ref{pr24}) follows from 
the definition of \(R_1,\ldots,R_{\mu+1}\). Suppose that \(\bk\in \Lambda_2\). 
By the definition of \(\mathcal{I}_2\), we have \(I(i)\subset J(j)\) for some \(j\) with \(1\leq j\leq \tau\), 
and so \(R\in (T_j,T_{1+j})\). Thus, we get (\ref{pr24}). \par
Applying the third assumption of Theorem \ref{thm:cri1} with \(\varepsilon=\delta/(2D)\), 
we see by \(g_1\leq A\) that 
\[\underline{\lambda}(N)^{\bk_1}=o\left(N^{-\delta/2+1}\right)\]
as \(N\in \Xi\) tends to infinity. Thus, we obtain for any sufficiently large \(N\in \Xi\) that 
\begin{align}
|I(i)|\geq \frac{1}{12C_{16}}\frac{N}{\underline{\lambda}(N)^{\bk_1}}\geq N^{\delta/2}.
\label{pr25}
\end{align}
We can apply 
the fourth assumption of Theorem \ref{thm:cri1} with 
\[R=\frac{|I(i)|}{1+C_8}\geq \frac{N^{\delta/2}}{1+C_8}\geq C_9\]
by (\ref{pr25}) and (\ref{pr23}). Thus, we get that 
there exists \(V(N,i)\in S_r\) with 
\[\frac{|I(i)|}{1+C_8}\leq V(N,i)\leq \frac{C_8|I(i)|}{1+C_8}.\]
Put 
\(M=M(N,i):=R_i+V(N,i)\). Then we have 
\begin{align}
R_i+\frac{|I(i)|}{1+C_8}\leq M\leq R_i+\frac{C_8|I(i)|}{1+C_8}.
\label{pr26}
\end{align}
By the definition of \(R_i\), there exists \(k_r\leq -1+g_r\) such that 
\[R_i\in\sum_{h=1}^{r-1} g_h S_h+k_r S_r.\]
Using Remark \ref{rem:cri1}, we see 
\[R_i\in\sum_{h=1}^{r-1} g_h S_h+(-1+g_r) S_r.\]
Thus, we get 
\begin{align}
M\in\sum_{h=1}^r g_h S_h
\label{pr27}
\end{align}
by \(V(N,i)\in S_r\). 
\begin{lem}
Let \(N\in\Xi\) be sufficiently large and let \(1\leq i\leq \mu\) with \(I(i)\in\mathcal{I}_2\). Then 
\(Y_R>0\) for any \(R\) with \(R_i\leq R<M\). 
\label{prlem8}
\end{lem}
\begin{proof}
We prove Lemma \ref{prlem8} by induction on \(R\). First we show that \(Y_{M-1}>0\). 
We see 
\begin{align}
Y_{M-1}&=
A_{\bg} \sum_{m=1}^{\infty}\beta^{-m}\rho(\bg;m+M-1)\nonumber\\
&\hspace{10mm}+\sum_{\bk\in\Lambda\backslash\{\bg\}}A_{\bk}
\sum_{m=1}^{\infty}\beta^{-m}\rho(\bk;m+M-1)\nonumber\\
&=:S^{(3)}+S^{(4)}.
\label{pr28}
\end{align}
By (\ref{pr27}) 
\begin{align}
S^{(3)}\geq \frac{A_{\bg}}{\beta}\rho(\bg;M)\geq \frac{1}{\beta}.
\label{pr29}
\end{align}
We now estimate upper bounds for \(|S_4|\). Let \(m\) be an integer with 
\begin{align}
1\leq m\leq -1+\lceil 2D\log_{\beta} N\rceil.
\label{eqn:abc}
\end{align}
Using (\ref{pr26}) and (\ref{pr25}), we get 
\begin{align*}
R_{i+1}-M&\geq R_{i+1}-R_i-\frac{C_8|I(i)|}{1+C_8}\\
&=\frac{|I(i)|}{1+C_8}>m
\end{align*}
for sufficiently large \(N\in \Xi\) and 
\[R_{i+1}>m+M-1>R_i.\]
Thus, applying (\ref{pr24}) with \(R=m+M-1\) for any \(m\) with (\ref{eqn:abc}), we obtain 
by (\ref{pr2}) that 
\begin{align*}
|S^{(4)}|&\leq
\sum_{\bk\in\Lambda\backslash\{\bg\}}|A_{\bk}|\sum_{m=\lceil 2D\log_{\beta} N\rceil}^{\infty}\beta^{-m}\rho(\bk;m+M-1)\\
&\leq
\sum_{\bk\in\Lambda\backslash\{\bg\}}|A_{\bk}|\sum_{m=\lceil 2D\log_{\beta} N\rceil}^{\infty}
\beta^{-m}C_7^D(m+N)^D\\
&\ll \sum_{m=\lceil 2D\log_{\beta} N\rceil}^{\infty}\beta^{-m}(m+N)^D.
\end{align*}
Therefore, (\ref{pr16}) implies that 
\[|S^{(4)}|\ll N^{-2D}\left(\lceil 2D\log_{\beta} N\rceil+N\right)^D\sum_{m=0}^{\infty}\beta^{m}
\left(\frac{1+\beta}{2}\right)^m=o(1)\]
as \(N\) tends to infinity. In particular, if \(N\in \Xi\) is sufficiently large, then 
\begin{align}
|S^{(4)}|<\frac1{2\beta}.
\label{pr30}
\end{align}
Combining (\ref{pr28}), (\ref{pr29}), and (\ref{pr30}), we deduce that 
if \(N\in \Xi\) is sufficiently large, then \(Y_{M-1}>0\). \par
Next, we assume that \(Y_R>0\) for some \(R\) with \(R_i<R<M\). Using (\ref{pr24}), we see 
\begin{align}
Y_{R-1}&=
\sum_{\bk\in\Lambda} A_{\bk}\frac1{\beta}\rho(\bk;R)+
\sum_{\bk\in\Lambda} A_{\bk}\sum_{m=2}^{\infty}\beta^{-m}\rho(\bk;m+R-1)\nonumber\\
&=\frac{A_{\bg}}{\beta}\rho(\bg;R)
+\frac1{\beta}\sum_{\bk\in\Lambda} A_{\bk}\sum_{m=1}^{\infty}\beta^{-m}\rho(\bk;m+R)\nonumber\\
&=
\frac{A_{\bg}}{\beta}\rho(\bg;R)+\frac1{\beta}Y_R.
\label{pr31}
\end{align}
By the inductive hypothesis 
\[Y_{R-1}>\frac{A_{\bg}}{\beta}\rho(\bg;R)\geq 0.\]
Therefore, we proved Lemma \ref{prlem8}. 
\end{proof}
Recall that \(\Omega\) is defined in (\ref{Omega}). 
\begin{lem}
Let \(N\in\Xi\) be sufficiently large and let \(1\leq i\leq \mu\) with \(I(i)\in\mathcal{I}_2\). Let \(R\) be an integer with 
\[R_i+4C_{13}\log_{\beta} N\leq R< M. \]
Then we have 
\begin{align}
R-\theta(R;\Omega)
\leq 2C_{13}\log_{\beta} N.
\label{pr32}
\end{align}
\label{prlem9}
\end{lem}
\begin{proof}
Put 
\(R_1:=\theta(R;\Omega)
.\)
In the same way as the proof of (\ref{pr31}), we see for any integer \(n\) with \(R_i<n<R_{i+1}\) that 
\begin{align}
Y_{n-1}=\frac{A_{\bg}}{\beta}\rho(\bg;n)+\frac1{\beta}Y_n.
\label{pr33}
\end{align}
First, we consider the case of \(Y_R\geq 1\). Then (\ref{pr33}) implies that 
\[Y_{R-1}\geq \frac{1}{\beta}\]
and that \(R-R_1=1\), which implies (\ref{pr32}). \par
In what follows, we may assume that \(0<Y_R<1\) by Lemma \ref{prlem8}. Let \(S:=\lceil C_{13} \log_{\beta} N\rceil\). 
Suppose for any integer \(m\) with \(0\leq m\leq S\) that 
\[\rho(\bg;R-m)=0.\]
Noting \(M>R>R-1>\cdots>R-S>R_i\), we get by (\ref{pr33}) that 
\[1>Y_R=\beta Y_{R-1}=\cdots=\beta^S Y_{R-S}=\beta^{1+S}Y_{R-S-1}>0,\]
where we use Lemma \ref{prlem8} for the last inequality by \(R_i<R-S-1<M\). So we get 
\[\beta^{S+1}<Y_{R-S-1}^{-1}=|Y_{R-S-1}|^{-1}.\]
Since 
\[R-S-1\geq 2 C_{13}\log_{\beta} N>C_{14}\]
for any sufficiently large \(N\), we apply Lemma \ref{prlem4} as follows: 
\[\beta^{S+1}<|Y_{R-S-1}|^{-1}\leq (R-S-1)^{C_{13}}<N^{C_{13}}. \]
Thus, we obtain 
\[\lceil C_{13} \log_{\beta} N\rceil+1=S+1<C_{13} \log_{\beta} N,\]
a contradiction. \par
Hence, there exists an integer \(m'\) with \(0\leq m'\leq S\) satisfying 
\(\rho(\bg;R-m')\geq 1\). 
Applying (\ref{pr33}) with \(n=R-m'\), we get by \(Y_{R-m'}>0\) that 
\[Y_{R-m'-1}\geq \frac{A_{\bg}}{\beta}\rho(\bg;R-m')\geq \frac{1}{\beta},\]
where for the last inequality we use (\ref{pr9}). Hence, we deduce that 
\[R-R_1\leq m'+1\leq 2C_{13}\log_{\beta} N. \]
\end{proof}
\begin{lem}
\[\limsup_{N\to\infty}\frac{y_N}{\log N}>0.\]
\label{prlem10}
\end{lem}
\begin{proof}
Let \(N\in \Xi\) be sufficiently large and let \(1\leq i\leq \mu\) with \(I(i)\in\mathcal{I}_2\). 
Note that 
\begin{align}
\lim_{N\to\infty}\frac{|I(i)|}{\log_{\beta}N}=\infty
\label{eqn:bcd}
\end{align}
by (\ref{pr25}). Combining (\ref{pr26}), (\ref{eqn:bcd}), and Lemma \ref{prlem9}, we see that there exists a constant \(C_{17}\) such that 
\[y_N(i)\geq C_{17}\frac{|I(i)|}{\log N}. \]
Therefore, using (\ref{pr21}) and (\ref{pr22}), we obtain 
\[y_N\geq \sum_{1\leq i\leq \mu\atop I(i)\in \mathcal{I}_2} y_N(i)\geq 
\sum_{I\in \mathcal{I}_2}C_{17}\frac{|I|}{\log N}\gg \frac{N}{\log N}.\]
\end{proof}
Finally, we deduce a contradiction from Lemma \ref{prlem5} and \ref{prlem10}, which proves 
Theorem \ref{thm:cri1}. 
\section*{Acknowledgements}
This work was supported by JSPS KAKENHI Grant Number 15K17505.

Hajime Kaneko\\
Institute of Mathematics, University of Tsukuba, 1-1-1\\
 Tennodai, Tsukuba, Ibaraki, 350-0006, JAPAN\\
Center for Integrated Research 
in Fundamental Science and Technology (CiRfSE) 
University of Tsukuba, 
Tsukuba, Ibaraki, 305-8571, JAPAN\\
e-mail: kanekoha@math.tsukuba.ac.jp
\end{document}